\documentclass{amsart}
\usepackage{amsmath,amssymb,amsxtra,amsfonts,amsthm,comment,url,microtype}

\usepackage{tikz} 
\usetikzlibrary{arrows,backgrounds,decorations,automata}

\theoremstyle{plain}
\newtheorem{theorem}{Theorem}
\newtheorem{lemma}[theorem]{Lemma}
\newtheorem{corollary}[theorem]{Corollary}
\newtheorem{proposition}[theorem]{Proposition}

\theoremstyle{definition}

\newcommand{\A}{\mathcal{A}}

\newcommand{\Ab}{{\bf A}}
\newcommand{\Xb}{{\bf X}}
\newcommand{\Yb}{{\bf Y}}
\newcommand{\Zb}{{\bf Z}}

\newcommand{\B}{\mathbb}
\newcommand{\C}{\mathcal}

\newcommand{\ga}{\alpha}
\newcommand{\gb}{\beta}
\newcommand{\gd}{\delta}
\newcommand{\eps}{\varepsilon}

\newcommand{\gl}{\lambda}

\newcommand{\gD}{\Delta}

\newcommand{\gS}{\Sigma}

\newcommand{\SG}{\langle \A \rangle}

\DeclareMathOperator{\spn}{span}
\DeclareMathOperator{\tr}{Tr}
\DeclareMathOperator{\GL}{GL}
\DeclareMathOperator{\E}{End}

\begin{document}

\title[Growth degree classification]{Growth degree classification for finitely generated semigroups of integer matrices}

\author{Jason P. Bell}
\address{Department of Pure Mathematics, University of Waterloo, Waterloo, Canada}
\email{jpbell@uwaterloo.ca}
\thanks{Research of J.~P.~Bell was supported by NSERC grant 31-611456, the research of M.~Coons was supported by ARC grant DE140100223, and the research of K.~G.~Hare was partially supported by NSERC}

\author{Michael Coons}
\address{School of Math.~and Phys.~Sciences\\
University of Newcastle\\
Callaghan\\
Australia}
\email{Michael.Coons@newcastle.edu.au}

\author{Kevin G. Hare}
\address{Department of Pure Mathematics, University of Waterloo, Waterloo, Canada}
\email{kghare@uwaterloo.ca}
\date{\today}


\begin{abstract} 
Let $\A$ be a finite set of $d\times d$ matrices with integer entries and let $m_n(\A)$ be the maximum norm of a product of $n$ elements of $\A$. In this paper, we classify gaps in the growth of $m_n(\A)$; specifically, we prove that $\lim_{n\to\infty} \log m_n(\A)/\log  n\in\B{Z}_{\geqslant 0}\cup\{\infty\}.$ This has applications to the growth of regular sequences as defined by Allouche and Shallit.
\end{abstract}

\maketitle

\section{Introduction}

Let $\A=\{ \Ab_1, \Ab_2, \dots, \Ab_m\}$ be a finite set of matrices 
    and let $\|\cdot\|$ be a matrix norm. 
Let $m_n(\A)$ denote the maximum norm of a product of $n$ elements of 
    $\A$; specifically, 
    $$m_n(\A):=\max_{1\leqslant i_1,\ldots,i_n\leqslant m}\|\Ab_{i_1}\dots\Ab_{i_n}\|.$$ 
Recall that the joint spectral radius of a $\A$ is 
    $\rho(\A):=\lim_{n\to\infty}|m_n(\A)|^{1/n}.$

The joint spectral radius was first introduced by Rota and Strang \cite{RS} in 1960; it arises naturally in a wide variety of areas.
We say that $\A$ satisfies the {\em finiteness property} if the limit $\rho(\A)$ is achieved
    by a finite sequence of matrices;
that is, there exists a finite sequence of matrices $A_{i_1}, \ldots, A_{i_k}\in\C{A}$ such that 
    $\rho(\A)=\lim_{n\to\infty}\|(A_{i_1} A_{i_2} \cdots A_{i_k})^n\|^{1/nk}.$
In 1995, Lagarias and Wang \cite{LW} conjectured that all $\A$ would satisfy the 
   finiteness property, though this is now known to be false.
Non-constructive counterexamples have been given by Bousch and Mairesse \cite{BM},
    Blondel, Theys and Valdimirov \cite{BTV} and Kozyakin \cite{Koz3}.
The first constructive counterexample was recently given by Hare, Morris, Sidorov and Theys \cite{HMST}.
For more details concerning the joint spectral radius see \cite{BN, Cicone, HMS, HS, Jungers, BJ, JPB2008, Theys}.

In 2005, Bell \cite{B2005} showed that {\em if $\A$ is a finite set of
    $d\times d$ complex matrices, then the growth of $m_n(\A)$ is either at 
    least exponential or it is bounded by a polynomial of degree $d-1$; moreover, 
    $m_n(\A)$ is bounded by a polynomial if and only if the joint spectral radius 
    of $\A$ is at most $1$.} 
This result exhibits a gap in the possible types of growth of $m_n(\A)$. 
For example, as Bell points out in his paper, it is impossible to find a finite set of 
    matrices $\A$ such that $m_n(\A)\sim e^{\sqrt{n}}.$ 

Bell's result immediately raises the question of lower gaps in the growth of $m_n(\A)$. 
In a recent paper \cite{BCH2014}, we provided such a gap result, showing that $m_n(\A)$ 
    is either bounded or grows at least linearly.

In this paper, we obtain a generalisation of our aforementioned result \cite{BCH2014} by 
    providing a complete classification of gaps in the growth of $m_n(\A)$ for integer 
    matrices. 
To state our result explicitly, we use the following definitions.

Let $\A$ be a finite non-degenerate set of matrices. Here, a set $\A$ is called non-degenerate if $m_n(\A) \not\to 0$.
That is, for all $N$ there exists an $n \geq N$ with $m_n(\A) \neq 0$.
We define the {\em growth degree of $\A$} as
\begin{equation}
{\rm GrDeg}(\A):=\lim_{n\to\infty} \frac{\log m_n(\A) }{\log  n}.
\label{eq:GrDeg}
\end{equation}

\noindent It is not immediately clear that the limit in \eqref{eq:GrDeg} is well defined, though this will be a consequence of our main theorem, which is the following classification.

\begin{theorem}\label{mainSG} 
If $\A$ is a finite non-degenerate set of $d\times d$ integer matrices, 
    then ${\rm GrDeg}(\A)\in\B{Z}_{\geqslant 0}\cup\{\infty\}.$ 
Moreover, $k:={\rm GrDeg}(\A)\in\B{Z}_{\geqslant 0}$ if and only if the 
    joint spectral radius of $\A$ equals one, and in this case, there are positive constants $C_1$ and $C_2$ such that
    $$C_1n^k \leqslant m_n(\A) \leqslant C_2 n^k$$ for all $n\geqslant 1$.
\end{theorem}

We point out that we in fact prove a much stronger result (Theorem \ref{morder}) that deals with semigroups of complex matrices that have the property that every nonzero eigenvalue of each matrix is a root of unity.
It is worth observing here that, in the case of integer matrices, if $m_n(\A) \not\to 0$, then we necessarily
    have $m_n(\A) \geqslant 1$, and hence the joint spectral radius in question is bounded
    below by $1$.
Moreover, if $\rho(\A) > 1$, then it is easy to see that ${\rm GrDeg}(\A) = \infty$.
So the interesting case is when $\rho(\A) = 1$.  

In addition to the gaps provided by Bell \cite{B2005}, our result provides gaps of 
    smaller and intermediate orders; for example, it is impossible to find a finite 
    set of integer matrices $\A$ such that $m_n(\A)\asymp n^\alpha$ for any 
    $\alpha\in\B{R}\setminus\B{Z}.$ 
Indeed, our result implies that either $m_n(\A)$ is bounded or there is a 
    real constant $c>0$ such that $m_n(\A)>cn$ for all $n$ sufficiently large. Moreover, if there is a real number $\alpha>0$ such that $m_n(\A)>cn^\alpha$ for some positive constant $c$, then our result implies that there is a constant $C>0$ such that $m_n(\A)>Cn^{\lceil \alpha\rceil}$, where $\lceil x\rceil$ denotes the smallest integer greater than or equal to the real number $x$.

Theorem \ref{mainSG} addresses a question of Jungers, Protasov, and Blondel 
    \cite[Problem~2]{JPB2008}, who asked: 
    {\em is it true that for any finite set of matrices $\A$, 
    the limit $$\lim_{n\to\infty} \frac{\log [\rho(\A)^{-n}m_n(\A)]}{\log n}$$ exists 
    and is always an integer? 
In particular, does this hold for nonnegative integer matrices?} 

Theorem \ref{mainSG} shows that, in the special case where $\A$ is a finite set of integer matrices, and $\rho(\A)=1$, the answer to the above question is `yes'. Note that our result does not require the matrices to have nonnegative integer values. The case of matrices with strictly positive integer values was considered by Jungers, Protasov, and Blondel \cite{JPB2008}.

Our original motivation for Theorem \ref{mainSG} was to prove the analogous result in the context of regular sequences as defined by Allouche 
    and Shallit \cite{AS1992}. 

Let $\gS_m=\{1,\ldots,m\}$ be a finite alphabet, let $R$ be a commutative ring, let $M$ be a finitely generated $R$-module and let $f:\gS_m^*\to M$. 
Let $f^u(w) := f(uw)$.
We say that $f$ is {\em $(R,m)$-regular} if,
    $\mathrm{span}\left\{\{f^u(w)\}_{w \in \gS_m^*}: u \in \gS_m^*\right\}$, 
    the $R$-module spanned by the maps $f^u(w)$, is finitely generated.  
    
Connecting $(R,m)$-regularity with more commonly regarded objects, 
    Allouche and Shallit \cite[Theorem~2.3]{AS1992} showed that in the case when $R=\mathbb{Z}$, 
    {\em the sequence of values of $f$ can be produced by a deterministic finite automaton 
    with output (that is, $f$ is $(\B{Z},m)$-automatic) if and only if 
    $f$ is $(\B{Z},m)$-regular and $\#\{f(\gS_m^*)\}$ is finite.} 
Moreover, connecting $(\B{Z},m)$-regularity to semigroups of matrices, they 
    showed \cite[Theorem 2.2]{AS1992} that {\em $f$ is $(\B{Z},m)$-regular if 
    and only if there exist positive integers $m$ and $d$, matrices 
    $\Ab_1,\ldots,\Ab_m\in\B{Z}^{d\times d}$, and vectors ${\bf v},{\bf w}\in\B{Z}^d$ 
    such that 
    $$f(w)={\bf w}^T \Ab_w {\bf v},$$ 
    where $\Ab_w:=\Ab_{i_1}\cdots\Ab_{i_s}$, when $w={i_1}\cdots {i_s}.$} 
In this way, there is a correspondence between finitely generated semigroups of 
    integer matrices and regular sequences. 

Using this correspondence, we are led to define the {\em growth degree} of a
    non-degenerate $f:\gS_m^*\to\B{C}$ as 
    $${\rm GrDeg}(f):=\limsup_{n\to\infty} \max_{\{w\in\gS^*:|w|=n\}}
        \frac{\log|f(w)|}{\log  n},$$ 
    where we have used $|\cdot|$ to denote both the length of a word and 
    the absolute value of a real number. 
Here $f$ is degenerate if $f(w) = 0$ for all $w$ sufficiently long.
We are taking $\log|0| = -\infty$, which is always less than any real number
    (and hence can only be attained by the $\limsup$ for a degenerate $f$, which
    we explicitly disallow).
In the context of $(\B{Z},m)$-regular sequences, we have the following result, which, in 
    view of the correspondence given by Allouche and Shallit, is a near-restatement 
    of Theorem \ref{mainSG}.

\begin{theorem}\label{main} 
   Let $f:\gS_m^*\to\B{Z}$ be $(\B{Z},m)$-regular and non-degenerate.  Then 
    ${\rm GrDeg}(f)\in\B{Z}_{\geqslant 0}\cup\{\infty\}.$ 
Moreover, ${\rm GrDeg}(f)\in\B{Z}_{\geqslant 0}$ if and only if $f\in\C{R}_0(\gS_m)$, where $\C{R}_0(\gS_m)$ is the algebra of sequences generated by automatic sequences with convolution as multiplication.
\end{theorem}

Note that $\C{R}_0(\gS_m)$ is a subset of the $(\B{C},m)$-regular sequences that plays the analogous role of the semigroups of matrices with $\rho(\A) = 1$.

In addition to the graded classification of growth types provided, Theorem \ref{main} can be seen as a lower bound version of a result of Allouche and Shallit \cite[Theorem~2.10]{AS1992}, which states that if $f$ is a $(\B{C},m)$-regular 
    sequence, then $f(w)=O(e^{c|w|})$ for some positive real number $c$.

As a final remark in this introduction, we relate our second result to the following canonical example. Let $\gS_1=\{1\}$. Then $f:\gS_1^*\to\B{Z}$ is $(\B{Z},m)$-regular if and only if $f$ satisfies a linear recurrence. The ring $\C{R}_0(\gS_1)$ is all sequences whose generating power series are rational functions in $\B{C}[[x]]$ with (possible) poles at zero and roots of unity, and moreover, this ring  is generated as a $\B{C}$-algebra by the eventually periodic sequences under the convolution product. Theorem \ref{main} generalises this well-known result about the one-variable case to the multivariable case, regarded in the sense of non-commutative rational functions of Berstel and Reutenauer \cite{BR2008}.

Our paper is organised as follows. In Section \ref{sec:class}, we prove Theorem \ref{mainSG}. We apply this to the case of $(\B{Z},m)$-regular sequences by proving Theorem \ref{main} in Section \ref{sec:app}.

\section{Classification of finite growth degrees}
\label{sec:class}

For a finite set of matrices $\A=\{\Ab_1,\Ab_2,\ldots,\Ab_m\}$, 
    we write $\SG$ for the semigroup generated by these matrices under 
    matrix multiplication.   
Here, our semigroups include an identity element, arising from the empty product
    of elements in $\A$.
Let $\C{U}$ be the set of all roots of unity.
We say that a matrix is {\em tame} if all eigenvalues of the matrix lie in 
    $\C{U} \cup \{0\}$.
We say that $\SG$ is a {\em tame semigroup} if all matrices in $\SG$ are tame.

To prove Theorem \ref{mainSG}, we require the following lemmas.
\begin{lemma}\label{LemKK} Let $K$ be a finitely generated extension of $\mathbb{Q}$ and let $d$ be a positive integer.  Then the collection $\mathcal{Y}$ of roots of unity $\omega$ such that $\omega$ is the root of a nonzero degree $d$ polynomial with coefficients in $K$ is finite.
\end{lemma}
\begin{proof}
By the primitive element theorem, we can write $K=\mathbb{Q}(t_1,\ldots ,t_s)(\alpha)$, where $t_1,\ldots ,t_s$ are algebraically independent over $\mathbb{Q}$ and $\alpha$ is algebraic over 
$\mathbb{Q}(t_1,\ldots ,t_s)$.  Let $F=\mathbb{Q}(t_1,\ldots ,t_s)$.  Then there is some natural number $m$ such that $[K:F]=m<\infty$.  By assumption, if $\omega\in \mathcal{Y}$ then $[K(\omega):K]\leqslant d$ and so $$[F(\omega):F]\leqslant [K(\omega):K][K:F] \leqslant md.$$  It follows that for $\omega\in \mathcal{Y}$ there exists a nonzero polynomial
$$F(x):=\sum_{i=0}^{md} p_i(t_1,\ldots ,t_s) x^i$$ with each $p_i\in \mathbb{Q}[t_1,\ldots ,t_s]$ such that 
$F(\omega)=0$.  Since $\mathbb{Z}^s$ is Zariski dense in $\mathbb{C}^s,$ there exist an $s$-tuple $(a_1,\ldots ,a_s)\in \mathbb{Z}^s$ and some $i$ such that $p_i(a_1,\ldots ,a_s)\neq 0$.  We note that $t_1,\ldots ,t_s$ are algebraically independent over $\mathbb{Q}(\omega)$ since they are algebraically independent over $\mathbb{Q}$, so we can specialise to obtain that
$\sum_{i=0}^{md} p_i(a_1,\ldots ,a_s) \omega^i = 0$.  In particular, we see that $[\mathbb{Q}(\omega):\mathbb{Q}]\leqslant md$ and so $\mathcal{Y}$ is finite since there are only finitely many roots of unity with this property.
\end{proof}
We next need a version of Burnside's Theorem for tame semigroups.
\begin{lemma}\label{LemA} Let $K$ be a field of characteristic zero, let $d$ be a positive integer, and let $\A:=\{\Ab_1,\Ab_2,\ldots,\Ab_m\}$ be a set of $d\times d$ matrices with entries in $K$. Suppose $\spn_{K}\SG=K^{d\times d}$ and $\SG$ is a tame semigroup. Then $\#\SG<\infty.$
\end{lemma}

\begin{proof}  
Pick $\Xb_1,\ldots,\Xb_{d^2}\in\SG$ such that $$\sum_{i=1}^{d^2}K\Xb_i=K^{d\times d}.$$ Since every matrix in $\SG$ has all eigenvalues in $\C{U}\cup\{0\}$, if $\Yb\in \SG,$ then we have that each of the eigenvalues of $\Yb$ is either zero or is a root of unity.
Furthermore each of these eigenvalues has the property that it is the root of a nonzero degree $d$ polynomial (the characteristic polynomial of $\Yb$) with coefficients in $K$.  Now let $\mathcal{Y}$ denote the set of elements $\omega$ of $\C{U}$ with the property that $\omega$ is a root of a nonzero degree $d$ polynomial with coefficients in $K$.  By Lemma \ref{LemKK}, $\mathcal{Y}$ is finite.  It follows that $$\mathcal{Y}_d:= \{ \omega_1+\cdots  + \omega_e \colon e\leqslant d\ \mbox{and}\ \omega_1,\ldots ,\omega_e\in \mathcal{Y}\}$$ is finite.

By construction, $\tr(\Yb)\in \mathcal{Y}_d$ for every $\Yb\in \SG$.   Let $\phi:\SG\to\mathcal{Y}_d^{d^2}$ be given by $$\phi(\Yb)=\left\{\tr({\bf YX}_i)\right\}_{i=1}^{d^2}.$$ We claim that $\phi$ is injective. To see this, suppose $\Yb,\Zb\in \SG$ and $\phi(\Zb)=\phi(\Yb)$, that is, $\tr((\Yb-\Zb)\Xb_i)=0$ for all $i=1,\ldots,d^2$. Since the $\Xb_i$ span $K^{d\times d}$, we have $\tr((\Yb-\Zb){\bf U})=0$ for every matrix ${\bf U}\in K^{d\times d}$ and this gives $\Yb-\Zb=0,$ so that $\Yb=\Zb$ and $\phi$ is injective.

Since $\phi$ injects $\SG$ into the finite set $\mathcal{Y}_d^{d^2}$, the lemma is proved.
\end{proof}

\begin{lemma}
Let $K$ be an algebraically closed field of characteristic zero, let $d$ be a positive integer, and let $\A:=\{\Ab_1,\Ab_2,\ldots,\Ab_m\}$ be a set of $d\times d$ matrices with entries in $K$ that generate an infinite tame semigroup.  Then there exists a matrix ${\bf U}\in {\rm GL}_d(K)$ and $e\in \{1,2,\ldots ,d-1\}$ such that for $i=1,\ldots ,m$, we have
$${\bf U}^{-1}\Ab_i{\bf U}=\left[\begin{matrix} {\bf B}_i & {\bf D}_i\\ {\bf 0}_{(d-e)\times e} & {\bf C}_i\end{matrix}\right],$$ where ${\bf B}_i\in K^{e\times e},$ ${\bf D}_i\in K^{e\times d},$ ${\bf C}_i\in K^{(d-e)\times (d-e)},$ and ${\bf 0}_{(d-e)\times e}$ is the $(d-e)\times e$ zero matrix.
\label{lem: GL}
\end{lemma}

\begin{proof} We prove this lemma by induction on $d$.  If $d=1$, then there do not exist infinite finitely generated tame semigroups of $d\times d$ matrices.  Thus we may assume that $d>1$.   Now assume that the conclusion holds for all dimensions less than $d$.  Let $\C{S}$ denote the $K$-span of $\SG$, and set $V=K^d$. 

If $V$ is a simple left $\C{S}$-module, then $\gD:=\E_\C{S}(V)$ is a division ring by Schur's lemma \cite[Page 356, Excerise 11]{DF2}.  Since any $\C{S}$-linear map from $V$ to $V$ is necessarily $K$-linear, we see that $\gD$ embeds in ${\rm End}_{K}(V)$ and so $\gD$ is finite-dimensional as a $K$-vector space and hence $\gD=K$ since $K$ is algebraically closed and hence has trivial Brauer group.  Then by the Jacobson Density Theorem \cite{JDT}, $\C{S}$ embeds as a dense subring of $\E_K(V)$. Since $\dim_{K}V<\infty$, we see that $\C{S}=\E_{K}(V)=K^{d\times d}.$   Thus by Lemma \ref{LemA}, we have that $\SG$ is finite, which contradicts our assumption. Thus we may assume that $V$ is not simple.

Since $V$ is not simple, there is an $\C{S}$-submodule $W$ with $0\subsetneq W\subsetneq V$. 

Let ${\bf u}_1,\ldots,{\bf u}_{d}$ be a $K$-basis for $K^d$ such that ${\bf u}_1,\ldots,{\bf u}_{e}$ is a basis for $W$ and the images of ${\bf u}_{e+1},\ldots,{\bf u}_{d}$ in $V/W$ form a basis for $V/W$.  Let $${\bf U}=\left[\begin{matrix}{\bf u}_1\ \cdots\ {\bf u}_{d}\end{matrix}\right]\in \GL_d(K).$$ Then for $i=1,\ldots,m$ we have $${\bf U}^{-1}\Ab_i{\bf U}=\left[\begin{matrix} {\bf B}_i & {\bf D}_i\\ {\bf 0}_{(d-e)\times e} & {\bf C}_i\end{matrix}\right],$$ where ${\bf B}_i\in K^{e\times e},$ ${\bf D}_i\in K^{e\times d},$ ${\bf C}_i\in K^{(d-e)\times (d-e)},$ and ${\bf 0}_{(d-e)\times e}$ is the $(d-e)\times e$ zero matrix, as claimed.
\end{proof}

The next result shows that if $m_n(\A)$ is polynomially bounded below 
    infinitely often, then it is polynomially bounded below on an arithmetic progression.

\begin{lemma}\label{gdef}
Let $\Xb_1,\ldots ,\Xb_k$ be tame $d\times d$ complex matrices, and 
    $\Zb_1,\ldots ,\Zb_k$ be $d\times d$ complex matrices.
Let ${\bf v}$ and ${\bf w}$ be vectors in $\B{C}^d$.  
Let $$g(n)={\bf w}^T \left( \prod_{i=1}^k \Xb_i^n \Zb_i\right){\bf v}.$$  
Then there is a positive integer $s$ such that for all $\ell \in \{0,\ldots ,s-1\}$, 
    the function $g(sn+\ell)$ is a polynomial in $n$.  
In particular, if there is a positive constant $C_0$ and a nonnegative integer 
    $r$ such that $|g(n)|>C_0 n^r$ for infinitely many $n$, then there is some 
    positive constant $C_1$ and some $\ell \in \{1,\ldots ,s-1\}$ such that 
    $|g(sn+\ell)|>C_1n^r$ for all $n\geqslant 0$.  
\label{lem: lr}
\end{lemma}

\begin{corollary}\label{marith}
Let $g(n)$ be as in the statement of Lemma \ref{gdef} and additionally suppose that $\Xb_1,\ldots ,\Xb_k,\Zb_1,\ldots ,\Zb_k\in\SG$.
If $g(s n + \ell) \geqslant c_2 n^k$ for $n$ sufficiently large, 
    then there exists $s'$ and $\ell'$ such that 
    $m_{s' n + \ell'}(\A) \geqslant c_2 n^k$ for all $n\geqslant 0$.
\end{corollary}

\begin{proof}
Notice that each $\Xb_i$ and $\Zb_i$ corresponds to a word over the alphabet $\C{A}$.
Taking the sums of the lengths of the words of $\Xb_i$ for $s_0$ and 
       the sums of the lengths of the words of $\Zb_i$ for $\ell_0$ gives
       $m_{s_0 n + \ell_0}(\A) \geqslant |g(n)|$.
By taking the appropriate subsequence the result follows.
\end{proof}

\begin{proof}[Proof of Lemma \ref{lem: lr}]
As the $\Xb_i$ are all tame, there exist natural 
    numbers $a,b,c$ with $a>b$ such that
    $(\Xb_i^a - \Xb_i^b)^c=0$ for all $i$.  
Let $d_i$ denote the smallest nonnegative integer for which we have
    $$(\Xb_i^a - \Xb_i^b)^{d_i}\Zb_i=0.$$
We proceed by induction on these $d_i$.

If any $d_i = 0$, then we have $\Zb_i={0}$ and $g(n)=0$, which is a polynomial in $n$.   
This proves the base case.

Assume that each $d_i>0$ and the inductive hypothesis is true for all 
     $(d_1', \cdots, d_k')$ where $d_i' \leqslant d_i$ for all $i$ with 
     strict inequality holding for at least one $i$. 

For $j=1,\ldots ,k$, we define 
    $$h_j(n)={\bf w}^T \left( \prod_{i<j} 
    \Xb_i^n \Xb_i^b \Zb_i\right)\left(\Xb_j^n (\Xb_j^a-\Xb_j^b)\Zb_j\right) 
    \left( \prod_{i>j} \Xb_i^n \Xb_i^a \Zb_i\right){\bf v}.$$
By telescoping, we have $g(n+a)-g(n+b) = \sum_{i=1}^k h_i(n)$.
Let $j\in \{1,\ldots ,k\}$ and define 
$$
\Zb_i' = \begin{cases}
    \Xb_i^b \Zb_i          &  \mathrm{if}\ i<j \\
    (\Xb_j^a -\Xb_j^b) \Zb_j & \mathrm{if}\ i = j \\
    \Xb_i^a\Zb_i           & \mathrm{if}\ i >j 
    \end{cases}. 
$$
By construction, we have 
    $$h_j(n) = {\bf w}^T \left( \prod_{i=1}^k \Xb_i^n \Zb_i'\right){\bf v}.$$ 
Moreover,  
    $$(\Xb_i^a - \Xb_i^b)^{d_i}\Zb_i'=0$$ for $i\neq j$, and since 
    $\Zb_j'$ contains a left factor of $(\Xb_j^a - \Xb_j^b)$, 
    we have $$(\Xb_j^a - \Xb_j^b)^{d_j-1}\Zb_j'=0.$$
In particular, by minimality of $(d_1,\ldots ,d_k)$, the function
$h_j(n)$ satisfies the conclusion of the lemma for $j=1,\ldots ,k$, and so
 there exist $s_1,\ldots ,s_k$ such that $h_j(s_j n +\ell_j)$ is a polynomial for $j=1,\ldots ,k$ and $\ell_j\in \{0,\ldots ,s_j-1\}$.  We now take $s=s_1\cdots s_k$ and obtain the desired result. 
\end{proof}

Corollary \ref{marith} implies, in a straightforward way, that $m_{n}(\A)$ is polynomially bounded below for all $n\geqslant 0$.
    
\begin{lemma}
\label{lem:lim}
Let $s \in \B{N}$, $\ell \in \{0, 1, \ldots, s-1\}$ and $c_1 > 0$ such 
    that for all $n\geqslant 0$ we have $m_{s n + \ell} (\A) \geqslant c_1 n^{k-1}$.
Then there exists a positive constant $c_2$ such that 
    $m_n(\A) \geqslant c_2 n^{k-1}$ for all $n\geqslant 0$.
\end{lemma}

\begin{proof}
Since our matrix norm is submultiplicative, for any three words $u,v,$ and $w$ such that $uv=w$, we have $$\|{\bf A}_w\|\leqslant \|{\bf A}_u\| \cdot \|{\bf A}_v\|,$$ and so $$m_{i+j}(\C{A})\leqslant m_i(\C{A})\cdot m_j(\C{A}),$$ for any nonnegative integers $i$ and $j$.

Let $n$ be an arbitrary given positive integer that may be taken to be sufficiently large. Define the positive integer $N$ by $$sN+\ell \geqslant n> sN+\ell-s.$$ Define the integer $r\in\{0,\ldots,s\}$ by $sN+\ell=n+r.$ Then we have $$m_{sN+\ell}(\A)\leqslant m_n(\A)\cdot m_r(\A).$$ The quantity $m_r(\A)$ is bounded by a constant, which is independent of the choice of $r$, but is dependent on $s$. Specifically $$m_r(\A)\leqslant M,$$ where $$M:=\max\left\{\max_{r\in\{0,\ldots,s\}}m_r(\A),1\right\}.$$ 

By the lower bound assumption on $m_{sN+\ell}(\A)$, $$m_n(\A)\geqslant \frac{m_{sN+\ell}(\A)}{M}\geqslant \frac{c_1N^{k-1}}{M}\geqslant\frac{c_1}{M}\left(\frac{n-\ell}{s}\right)^{k-1}.$$ This implies the result for $n$ large enough, and by adjusting the constant as necessary, for all $n\geqslant 0$.
\end{proof}

\begin{lemma}\label{nil} Suppose that $\C{A}:=\{\Ab_1,\ldots,\Ab_m\}$ is a set of $d\times d$ complex matrices that generate a tame semigroup.  Then there exists $a>0$ (depending only on $\A$ and $d$) and an integer $k$, with $1\leqslant k\leqslant d$, such that $$\prod_{i=1}^k (\Xb_i^{2ar}-\Xb_i^{ar})\Yb_i=0$$ for all $\Xb_1,\ldots,\Xb_k,\Yb_1,\ldots,\Yb_k\in\SG$ and all $r\geqslant 1$.
\end{lemma}

\begin{proof}  If $\SG$ is finite, then there exists some $a$ such that $\Xb^{2a}=\Xb^a$ for all $\Xb\in \SG$.  Thus we obtain the desired conclusion.

By Lemma \ref{lem: GL}, if $\SG$ is infinite, then there is a ${\bf U}\in\GL_d(\B{C})$ such that $${\bf U}^{-1}\Ab_i{\bf U}=\left(\begin{matrix} {\bf B}_i & {\bf D}_i\\ {\bf 0}& {\bf C}_i\end{matrix}\right),$$ where the eigenvalues of the elements of $\langle{\bf B}_1,\ldots,{\bf B}_m\rangle$ and $\langle{\bf C}_1,\ldots,{\bf C}_m\rangle$ are all in $\C{U}\cup\{0\}$.

By induction there are $a_1,a_2>0$ and $k_1$ and $k_2$ with $k_1+k_2\leqslant d$ such that $$\prod_{i=1}^{k_1} (\Xb_i^{2a_1r}-\Xb_i^{a_1r})\Yb_i=0$$ for all $\Xb_1,\ldots,\Xb_{k_1},\Yb_1,\ldots,\Yb_{k_1}\in\langle {\bf B}_1,\ldots,{\bf B}_m\rangle$, and $$\prod_{i=1}^{k_2} (\Xb_{k_1+i}^{2 a_2r}-\Xb_{k_1+i}^{a_2r})\Yb_{k_1+i}=0$$ for all $\Xb_{k_1+1},\ldots,\Xb_{k_1+k_2},\Yb_{k_1+1},\ldots,\Yb_{k_1+k_2}\in\langle {\bf C}_1,\ldots,{\bf C}_m\rangle$. 

Let $a:={\rm lcm}(a_1,a_2)$. If $\Xb_1,\ldots,\Xb_{k_1+k_2},\Yb_1,\ldots,\Yb_{k_1+k_2}\in\SG$, then $${\bf U}^{-1}\left(\prod_{i=1}^{k_1} (\Xb_i^{2ar}-\Xb_i^{ar})\Yb_i\right){\bf U}=\left(\begin{matrix} {\bf 0} & *\\ {\bf 0}& *\end{matrix}\right)$$ and $${\bf U}^{-1}\left(\prod_{i=k_1+1}^{k_1+k_2} (\Xb_i^{2ar}-\Xb_i^{ar})\Yb_i\right){\bf U}=\left(\begin{matrix} *& *\\ {\bf 0}& {\bf 0}\end{matrix}\right).$$ 
This gives $${\bf U}^{-1}\left(\prod_{i=1}^{k_1+k_2} (\Xb_i^{2ar}-\Xb_i^{ar})\Yb_i\right){\bf U}=\left(\begin{matrix} {\bf 0} & *\\ {\bf 0}& *\end{matrix}\right)\left(\begin{matrix} * & *\\ {\bf 0}& {\bf 0}\end{matrix}\right)={\bf 0},$$ which is the desired result.
\end{proof}

We use Lemma \ref{nil} to establish the following result that directly implies Theorem~\ref{mainSG}.

\begin{theorem} \label{morder}
    Let $\A:=\{\Ab_1,\ldots,\Ab_m\}$ be a non-degenerate set of $d\times d$ complex 
    matrices that generate a tame semigroup and let $k$ be the minimal nonnegative integer for which 
    there exists an $a$ such that
    $$\prod_{i=1}^{k} (\Xb_i^{2a}-\Xb_i^{a})\Yb_i=0,$$ 
    for all $\Xb_1,\ldots,\Xb_{k},\Yb_1,\ldots,\Yb_{k}\in\SG$.  
Then there exist positive constants $C_1$ and $C_2$ such that
    $$C_1 n^{k-1} \leqslant m_n(\C{A}) \leqslant C_2 n^{k-1}$$ for all $n\geqslant 0$.
\end{theorem}
We note that a minimal $k$ necessarily exists by Lemma \ref{nil}.
\begin{proof}[Proof of Theorem \ref{morder}] 
If $\SG$ is a finite non-degenerate semigroup, we have $m_n(\A) = 1$ for
    all $n$.
Hence the result holds for $k = C_1 = C_2 = 1$.
Our proof follows by induction on $d$. 
When $d=1$, our semigroup is a finite non-degenerate semigroup, 
    hence the result follows.
Assume that the result is true for all dimensions less than $d$. 

By assumption, there is an $a>0$ and a $k\geqslant 1$ such that 
    $$\prod_{i=1}^{k} (\Xb_i^{2a}-\Xb_i^{a})\Yb_i=0$$ 
    for all $\Xb_1,\ldots,\Xb_k,\Yb_1,\ldots,\Yb_k\in\SG.$ 
By expanding over products, we see that
    $$\prod_{i=1}^{k} (\Xb_i^{2a}-\Xb_i^{a})\Yb_i=0$$ 
    for all $\Xb_1,\ldots,\Xb_k \in \SG$ and 
    $\Yb_1,\ldots,\Yb_k\in \spn_{\B{C}}\SG$.

By the minimality of $k$, there exist (fixed) 
    $$\Xb_1,\ldots,\Xb_{k-1},\Yb_1,\ldots,\Yb_{k-1}\in\SG$$ 
    such that 
    \begin{equation}{\bf C}
    :=\prod_{i=1}^{k-1} (\Xb_i^{4a}-\Xb_i^{2a})\Yb_i\neq 0.\end{equation} 

Now, in general, for $n\geqslant 4$, 
    \begin{align}\label{xnaxa} 
    \Xb^{na}-\Xb^a=(\Xb^{2a}-\Xb^a)^2&\left(\sum_{i=0}^{n-4}(n-3-i)\Xb^{ai}\right) \\ \nonumber &
    + n(\Xb^{3a}-\Xb^{2a}) - (\Xb^{2a}-\Xb^a)(2\Xb^a -{\bf I}).\end{align} 
Let $$P_{1,n}(\Xb):=\left(\sum_{i=0}^{n-4}(n-3-i)\Xb^{ai}\right)(\Xb^a+{\bf I})$$ and let
    $$P_{2,n}(\Xb):=(2\Xb^a-{\bf I})(\Xb^a+{\bf I}).$$ 
Then
     \begin{equation}\label{xnaxa2} (\Xb^{na}-{\bf X}^a)(\Xb^a+{\bf I})
     =(\Xb^{2a}-\Xb^a)^2P_{1,n}(\Xb) + n(\Xb^{4a}-\Xb^{2a})-(\Xb^{2a}
      -\Xb^a)P_{2,n}(\Xb).\end{equation} 
So by \eqref{xnaxa2}, 
    $\prod_{i=1}^{k-1} (\Xb_i^{na}-\Xb_i^{a})(\Xb_i^a+{\bf I})\Yb_i$ 
    is equal to 
    $$\prod_{i=1}^{k-1}  \Big[ (\Xb_i^{2a}-\Xb_i^a)^2P_{1,n}(\Xb_i) 
    + n(\Xb_i^{4a}-\Xb_i^{2a}) - (\Xb_i^{2a}-\Xb_i^a)P_{2,n}(\Xb_i) \Big]\Yb_i.$$

When we expand this product out, any term consisting of a product containing 
    $(\Xb_i^{2a}-\Xb_i^a)^2 P_{1,n}(\Xb_i)$ will necessarily be zero, since it 
    will contain at least $k$ factors of the form $\Xb^{2a}-\Xb^a$.  
Thus
\begin{align}
    \prod_{i=1}^{k-1} (\Xb_i^{na}-\Xb_i^{a})(\Xb_i^a+{\bf I})\Yb_i
\nonumber    &=\prod_{i=1}^{k-1}  \left( n(\Xb_i^{4a}-\Xb_i^{2a}) 
       - (\Xb_i^{2a}-\Xb_i^a)P_{2,n}(\Xb_i) \right)\Yb_i\\
\nonumber    &=n^{k-1}\prod_{i=1}^{k-1}(\Xb_i^{4a}-\Xb_i^{2a}) \Yb_i+O(n^{k-2}), \\
\label{XXYC}    &=n^{k-1} {\bf C} +O(n^{k-2}),
\end{align} 
    for the fixed matrices $\Xb_1,\ldots,\Xb_{k-1},\Yb_1,\ldots,\Yb_{k-1}.$ 
By our choice of $\Xb_i$ and $\Yb_i$, we have that ${\bf C} \neq 0$.

Notice that $\prod_{i=1}^{k-1} (\Xb_i^{na}-\Xb_i^{a})(\Xb_i^a+{\bf I})\Yb_i$ is a 
    $\{\pm 1\}$-linear combination of $4^{k-1}$ elements of $\SG$, of the form
    $\pm \Xb_1^{\alpha_1} \Xb_1^{\beta_1} \Yb_1  
         \Xb_2^{\alpha_2} \Xb_2^{\beta_2} \Yb_2 \dots
         \Xb_{k-1}^{\alpha_{k-1}} \Xb_{k-1}^{\beta_{k-1}} \Yb_{k-1}$
    where $\alpha_i \in \{na, a\}$ and  $\beta_i \in \{a, 0\}$.
In particular, since ${\bf C}$ is nonzero, there exist nonzero vectors ${\bf v}$ and ${\bf w}$ 
    of norm $1$, a positive constant $\kappa_1$, some subset 
    $T\subseteq \{1,\ldots, k-1\}$, and matrices 
    $\{\Zb_i\colon i\in T\}\subseteq \SG$ such that 
    $$g(n) := {\bf w}^T \left( \prod_{i\in T} \Xb_i^{na}\Zb_i\right) {\bf v}$$ 
    has absolute value at least $\kappa_1 n^{k-1}$ for infinitely many $n$.  
By Lemma \ref{lem: lr}, there exist natural numbers $\ell$ and $s$ and a positive constant 
    $\kappa_2$ such that 
    $|g(sn+\ell)|\ge \kappa_2 n^{k-1}$ for all $n\geqslant 0$.
By Lemma \ref{lem:lim}, there exists a $C_1 > 0$ such that 
    $|m_n(\A)| \geqslant C_1 n^{k-1}$ for all $n\geqslant 0$.

We now give the corresponding upper bound to prove the theorem.  
If $k=1$, then our semigroup is finite and there is nothing to prove.  
As before if $k>1$, then $\SG$ is infinite, and so by Lemma \ref{lem: GL}, $V$ is not simple.
Let 
    $$V_0=\sum \B{C} \Yb(\Xb^{2a}-\Xb^a)V,$$ 
    where $\Xb$ and $\Yb$ range over $\SG$. Note that $V_0$ is invariant under $\SG$.
Then by construction, we have
    $$\prod_{i=1}^{k-1} (\Xb_i^{2a}-\Xb_i^{a})\Yb_i$$  
    is identically zero on $V_0$. 
In particular, we have by induction on $k$ that $\SG$ restricted to $V_0$ is such that 
    there exists a $\kappa_3 > 0$ with $m_n(\A|_{V_0}) \leqslant \kappa_3 n^{k-2}$.

We see for all $\Xb \in \A$ that $\Xb^{2a}-\Xb^a$ maps $V$ to $V_0$.
Hence, the image of $\SG$ in ${\rm End}(V/V_0)$ is a finitely generated, periodic, linear semigroup, and so the image is finite.

As the image of the semigroup is finite, we see that there exists an $M$ 
    such that if $|w|\geqslant M$, then there is a word $w_0$ with $|w_0|<M$ such that $$(\Ab_w-\Ab_{w_0})V\subseteq V_0.$$  Moreover, we have that 
\begin{equation}
\label{eq: norm}
\|\Ab_w|_{V_0}\|\leqslant \kappa_3|w|^{k-2}
\end{equation} for all non-trivial words $w$.  
Here $\Ab_w = \Ab_{i_1}\cdots\Ab_{i_s}$ where $w=i_1\cdots i_s \in \{1,\ldots ,m\}^*$.

Let $w$ be such that $m_n(\A) = \|\Ab_w\|$.
Write $w_0=v_0 u_0$ with  $|u_0|=M$. 
Then there is some $u_1$ with $|u_1|<M$ such that $(\Ab_{u_0}-\Ab_{u_1})V\subseteq V_0$.
Notice that
\begin{align*}
m_n(\A)  &=   \| \Ab_w\| \\
         &=   \| \Ab_{v_0 u_0} - \Ab_{v_0 u_1} + \Ab_{v_0 u_1}\| \\
         &\leqslant   \| \Ab_{v_0 u_0} - \Ab_{v_0 u_1}\| + \|\Ab_{v_0 u_1}\|  \\
         &\leqslant   \| \Ab_{v_0} (\Ab_{u_0} - \Ab_{u_1})\| + \|\Ab_{v_0 u_1}\| 
\end{align*}
We see that $\|\Ab_{u_0} - \Ab_{u_1}\|$ takes a unit vector from $V$ and takes it to $V_0$.
There are only a finite number of words of length $M$, and hence we can define
    $\kappa_4 = \max_{|u_0| = M, |u_1|<M} \|\Ab_{u_0} - \Ab_{u_1}\|$, such that
    $(\Ab_{u_0} - \Ab_{u_1}) V \subset V_0$.
By our comments before, we have by induction that $\Ab_{v_0}$ acting on $V_0$ satisfies the 
    inductive hypotheses, and hence satisfies $\|\Ab_{v_0}|_{V_0}\| \leqslant \kappa_3 n^{k-2}$  since $|v_0|\leqslant n$.
Lastly, we see that $|v_0 u_1| < |v_0 u_0|$ and hence $|v_0 u_1| \leqslant n-1$.
This implies that $\|\Ab_{v_0 u_1} \| \leqslant m_{n-1} (\A)$.
This gives
\begin{align*}
m_n(\A) & \leqslant   \kappa_3 \kappa_4 n^{k-2} + m_{n-1} (\A) \\
        & \leqslant   \kappa_3 \kappa_4 (n^{k-2} + (n-1)^{k-2}) + m_{n-2} (\A) \\
        & \leqslant   \kappa_3 \kappa_4 (n^{k-2} + (n-1)^{k-2} + (n-2)^{k-2}) + m_{n-3} (\A) \\
        & \ \ \vdots   \\
        & \leqslant   \kappa_3 \kappa_4 (n^{k-2} + (n-1)^{k-2} + \cdots + (M+1)^{k-2}) + m_{M} (\A) \\
        & \leqslant   C n^{k-1} 
\end{align*}
for some positive constant $C$ and for all $n\geqslant M$.  By replacing $C$ by some larger constant $C_2$ we see that we have that $m_n(\C{A}) \leqslant C_2 n^{k-1}$ for all $n\geqslant0.$
\end{proof}

\section{Application to regular sequences}\label{sec:app}

Recall from the introduction that for a finite alphabet $\gS_m:=\{1,\ldots,m\}$, we say that a function $f:\gS_m^*\to R$ is $(R,m)$-regular if the $R$-module spanned by the maps 
    $f^u(w):=f(uw)$ is finitely generated. Also recall (cf. \cite[Theorem 2.2]{AS1992}) that $f$ is $(R,m)$-regular if 
    and only if there exist positive integers $m$ and $d$, matrices 
    $\Ab_1,\ldots,\Ab_m\in R^{d\times d}$, and vectors ${\bf v},{\bf w}\in R^d$ 
    such that 
    $$f(w)={\bf w}^T \Ab_w {\bf v},$$ 
    where $\Ab_w:=\Ab_{i_1}\cdots\Ab_{i_s}$, when $w={i_1}\cdots {i_s}.$ In fact, if $R$ is a PID then after fixing an $R$-module basis for the module spanned by the maps of the form $f^u$ there is a canonical integer $d$ and choice of vectors ${\bf v},{\bf w}\in R^d$ and matrices 
    $\Ab_1,\ldots,\Ab_m\in R^{d\times d}$. We give this construction in the following lemma. 
    
As before, for a finite set of matrices $\A:=\{\Ab_1,\Ab_2,\ldots,\Ab_m\}$, we write 
$\SG$ for the semigroup generated by these matrices under matrix multiplication. 
Here again, our semigroups will include an identity element, arising from the empty product
    of elements in $\A$.  We note that we include fields as PIDs.

\begin{lemma}\label{canonical}  Let $R$ be a subring of $\B{C}$. Suppose that $R$ is a principal ideal domain, that $f$ is $(R,m)$-regular, and that the $R$-module spanned by the maps $f^u$ has $R$-module basis $\{g_1,\ldots,g_d\}$. Then there is a canonical choice (with respect to our chosen basis) of vectors ${\bf v},{\bf w}\in R^d$ and matrices $\Ab_1,\ldots,\Ab_m\in R^{d\times d}$ such that $\spn_{\mathbb{C}}{\bf w}^T\SG=\mathbb{C}^{1\times d}$. In particular, one can take ${\bf w}^T=[g_1(\eps),\ldots,g_d(\eps)],$ where $\eps$ is the empty word.
\end{lemma}

\begin{proof} Since $R$ is a PID and the $R$-module spanned by maps of the form $f^u$ is finitely generated and torsion free, we see that it has an $R$-module basis.  Let $\{g_1(w),\ldots,g_d(w)\}$ be an $R$-module basis for the $R$-module spanned by the maps $f^u(w)$. Then for $i\in\Sigma_m$ the functions $g_1(iw),\ldots,g_d(iw)$ can be expressed as $R$-linear combinations of $g_1(w),\ldots,g_d(w)$ and hence there are $d\times d$ matrices ${\bf A}_{1},\ldots,{\bf A}_{m}$ with entries in $R$ such that $$[g_1(w),\ldots,g_d(w)]{\bf A}_{i}=[g_1(iw),\ldots,g_d(iw)]$$ for $i=1,\ldots,m$ and all $w\in\Sigma_m^*.$ In particular, if we write $\eps$ for the empty word, and let $w\in\Sigma^*$ be any nonempty word where $w={i_s}\cdots {i_1}$ with ${i_1},\ldots, {i_s}\in\Sigma_m$, then $$[g_1(\eps),\ldots,g_d(\eps)]{\bf A}_{{i_s}}\cdots {\bf A}_{{i_1}}=[g_1(w),\ldots,g_d(w)].$$ We claim that the $\mathbb{C}$-span of the vectors $[g_1(w),\ldots ,g_d(w)]$, as $w$ ranges over all words in $\Sigma_m^*$, must span all of $\mathbb{C}^{1\times d}$.  Indeed, if this were not the case, then their span would be a proper subspace of $\mathbb{C}^{1\times d}$ and hence the span would have a non-trivial orthogonal complement.  In particular, there would exist $c_1,\ldots ,c_d\in R$, not all zero, such that $$c_1g_1(w)+\cdots +c_dg_d(w)=0$$ for every $w$, contradicting the fact that $g_1(w),\ldots ,g_d(w)$ are $R$-linearly independent sequences. 

Choosing ${\bf v}=[a_1,\ldots,a_d]^T\in R^d$ to be the unique vector such that $a_1g_1+\cdots+a_dg_d=f,$ finishes the proof of the lemma.
\end{proof}

We call the construction for $f$ in Lemma~\ref{canonical}, the {\em canonical representation of $f$}. Note that even though we call this representation `canonical', it is only unique up to conjugation by elements of $\GL_n(R)$ and that is why we first fix a basis.

We define the set $\C{S}_0(\gS_m)$ to be the set of all maps $f:\Sigma_m^*\to \B{C}$ such that there is a positive integer $d$, matrices $\A = \{\Ab_1,\ldots,\Ab_m\}$ with 
    $\Ab_i \in\B{C}^{d\times d}$, and vectors ${\bf v},{\bf w}\in\B{C}^d$ such that $f(w)={\bf w}^T \Ab_w {\bf v},$ where $\Ab_w=\Ab_{i_1}\cdots\Ab_{i_s}$, when $w={i_1}\cdots {i_s}$ and $\SG$ a tame semigroup. That is, the set $\C{S}_0(\gS_m)$ is the set of all $(\B{C},m)$-regular functions $f$ whose associated set of matrices $\C{A}$ generates a tame semigroup $\SG$.

Recall that for a word $w\in\gS_m^*$, we denote by $|w|$ the length of the word $w$; we also use $|x|$ for the standard absolute value of the real number $x$. The contexts of these usages are quite evident and not easily confused in what follows. 

The set of $(\mathbb{C},m)$-regular sequences exhibits algebraic structure. 
Given $f,g:\gS_m^*\to\B{C}$, both $(\B{C},m)$-regular, we define the 
    {\em convolution product} $f\star g:\gS_m^*\to\B{C}$ by 
    $$(f\star g)({i_1}\cdots {i_s}):=\sum_{j=0}^s 
       f({i_1}\cdots {i_j})g({i_{j+1}}\cdots {i_s}).$$ 
It is well-known (cf. \cite[Theorem 3.1 and Corollary 3.2]{AS1992}) that the set of $(\mathbb{C},m)$-regular 
    sequences forms a ring under pointwise addition and convolution product, 
    which we denote by $\C{R}(\gS_m).$ 
We denote by $\C{R}_0(\gS_m)$ the $\B{C}$-subalgebra of $\C{R}(\gS_m)$ generated 
    by the $(\B{C},m)$-automatic sequences under the convolution product.  (An automatic sequence is a regular sequence taking only finitely many distinct values.)
    
We say a $(\mathbb{C},m)$-regular function $f:\gS_m^*\to\B{C}$ is {\em polynomially bounded}, provided there is a $k\geqslant 0$ and $c>0$ such that $|f(w)|\leqslant c\cdot|w|^k.$

To aid in the proof of Theorem \ref{main}, we prove the following equivalence.

\begin{theorem}\label{iff} Let $f:\gS_m^*\to\B{Z}$ be $(\B{Z},m)$-regular. Then the following are equivalent:
\begin{enumerate}
\item[(i)] $f\in\C{R}_0(\gS_m),$
\item[(ii)] $f\in\C{S}_0(\gS_m),$
\item[(iii)] $f$ is polynomially bounded.
\end{enumerate}
\end{theorem}

We prove Theorem \ref{iff} by first considering the equivalence of the statements (i) and (ii)---this does not require that $f$ be $\mathbb{Z}$-valued and we show in fact that $\C{R}_0(\gS_m)=\C{S}_0(\gS_m)$.  We then show the equivalence of (ii) and (iii), which does require that $f$ be $\mathbb{Z}$-valued.  For example, if $f: \gS_1^*\to \B{C}$ is $(\B{C},m)$-regular and defined by $f(1^n)=(1/2)^n$, then $f$ is polynomially bounded, but it is not in $\C{S}_0(\gS_1)$.  

We make use of the following fact.   

\begin{lemma}\label{gsauto}
The set $\C{S}_0(\gS_m)$ contains all $(\B{C},m)$-automatic functions. 
\end{lemma}

\begin{proof}
Let $f:\gS_m^*\to \mathbb{C}$ be a $(\mathbb{C},m)$-automatic function and let $\{g_1\ldots ,g_d\}$ be an $R$-module basis for the $R$-module spanned by maps of the form $f^u$, and let $\SG$ be the semigroup of the set of matrices determined by the canonical representation of $f$ and ${\bf w}^T$ be also as given by the canonical representation. By Lemma \ref{canonical}, $\spn_{\B{C}}{\bf w}^T\langle \C{A}\rangle=\B{C}^{1\times d}$.
     
It suffices to show that $\SG$ is a tame semigroup.
To this end, suppose that $\Ab \in \SG$ and that $\gl$ is an eigenvalue of $\Ab$. Then there is a nonzero vector ${\bf u}$ such that $\Ab{\bf u}=\lambda {\bf u}$.  Let ${\bf y}$ be a vector such that ${\bf y}^T{\bf u}\neq 0$. Note that one can choose ${\bf y}$ to be the vector whose $i$th entry is the complex conjugate of the $i$th entry of ${\bf u}$.

Since $\spn_{\B{C}}{\bf w}^T\langle \C{A}\rangle=\B{C}^{1\times d}$, there exist $c_1,\ldots,c_\ell\in \B{C}$ and ${\bf X}_i\in \SG$ such that ${\bf y}^T=\sum_{i=1}^\ell c_i {\bf w}^T  {\bf X}_i$. Thus $${\bf y}^T \Ab^n = \sum_{i=1}^\ell c_i {\bf w}^T {\bf X}_i \Ab^n.$$ For $i=1,\ldots,\ell$ let $x_i\in\Sigma_m^*$ be the word corresponding to ${\bf X}_i$ and let $a\in\Sigma_m^*$ be the word corresponding to $\Ab$; that is $x_i=i_s\cdots i_0$, where ${\bf X}_i=\Ab_{x_i}=\Ab_{i_s}\cdots \Ab_{i_0}$, and similarly for $a$. Then $${\bf w}^T{\bf X}_i\Ab^n={\bf w}^T\Ab_{x_ia^n}=[g_1(x_ia^n),\ldots,g_d(x_ia^n)].$$ 
Write ${\bf u}=[b_1,\ldots,b_d]^T.$ Then 
\begin{equation}\label{lambdabig}\gl^n\cdot{\bf y}^T{\bf u}={\bf y}^T\gl^n{\bf u}={\bf y}^T \Ab^n {\bf u}=\sum_{i=1}^\ell\sum_{j=1}^dc_ib_jg_j(x_ia^n).\end{equation}

By assumption, each of $\{g_1(n)\}_{n\geqslant 0},\ldots ,\{g_d(n)\}_{n\geqslant 0}$ is in the $\mathbb{C}$-module generated by the maps $f^u$, and since $f$ is automatic, each of the maps $f^u$ can take only a finite number of values, so that also each of the functions $g_i$ can take only a finite number of values. Thus the sum on the righthand side of \eqref{lambdabig} can take only a finite number of values as $n\to\infty$, and hence also the lefthand side. This is only possible if $\gl$ is an element of $\C{U}\cup\{0\}.$  Thus $\SG$ must be tame.
\end{proof}

\begin{proposition}\label{Thm1} We have $\C{R}_0(\gS_m)=\C{S}_0(\gS_m)$.
\end{proposition}

\begin{proof} We first prove the inclusion $\C{R}_0(\gS_m)\supseteq\C{S}_0(\gS_m)$. Let $f\in\C{S}_0(\gS_m)$, so that $f(w)={\bf w}^T \Ab_w {\bf v}$, where ${\bf w},{\bf v}$ and $\C{A}$ are as provided by the definition of $\C{S}_0(\gS_m)$ and the semigroup $\SG$ is tame.  We let $d$ be the natural number for which the elements of our semigroup lie in $\B{C}^{d\times d}$, let $\C{S}$ denote the $\B{C}$-span of $\SG$, and set $V=\B{C}^d$.  We claim that $f\in \C{R}_0(\gS_m)$.  We prove this by induction on $d$.  

Notice that if $\C{A}$ is finite, then $f$ is automatic, so $f\in \C{R}_0(\gS_m)$ and we are done.  In particular, this occurs when $d=1$ and so we obtain the base case. Now suppose that the claim holds for all dimensions less than $d$.  We may assume that $\C{A}$ is infinite.   Then by Lemma \ref{lem: GL}, there is a matrix ${\bf U}$ and $e\in \{1,\ldots ,d-1\}$ such that
for $i=1,\ldots,m$ we have $${\bf U}^{-1}\Ab_i{\bf U}=\left[\begin{matrix} {\bf B}_i & {\bf D}_i\\ {\bf 0}_{(d-e)\times e} & {\bf C}_i\end{matrix}\right],$$ where ${\bf B}_i\in\B{C}^{e\times e},$ ${\bf D}_i\in\B{C}^{e\times d},$ ${\bf C}_i\in\B{C}^{(d-e)\times (d-e)},$ and ${\bf 0}_{(d-e)\times e}$ is the $(d-e)\times e$ zero matrix.

If $w={i_1}\cdots {i_s}$, then $${\bf U}^{-1}\Ab_w{\bf U}=\left[\begin{matrix} {\bf B}_w & \sum_{j=0}^{s-1}{\bf B}_{i_1}\cdots {\bf B}_{i_j}{\bf D}_{i_{j+1}}{\bf C}_{i_{j+2}}\cdots{\bf C}_{i_s}\\ {\bf 0}_{(d-e)\times e} & {\bf C}_w\end{matrix}\right],$$ where ${\bf B}_w={\bf B}_{i_1}\cdots {\bf B}_{i_s}$ and ${\bf C}_w={\bf C}_{i_1}\cdots {\bf C}_{i_s}$.

Notice that all the eigenvalues of the matrices in $\langle {\bf B}_{1},\ldots,{\bf B}_m\rangle$ and $\langle {\bf C}_{1},\ldots,{\bf C}_m\rangle$ are in $\C{U}\cup\{0\}$. Let ${\bf y}^T={\bf w}^T{\bf U}$, ${\bf x}={\bf U}^{-1}{\bf v}$, and write $${\bf x}=\left[\begin{matrix} {\bf x}_1\\ {\bf x}_2\end{matrix}\right]\quad\mbox{and}\quad {\bf y}=\left[\begin{matrix} {\bf y}_1\\ {\bf y}_2\end{matrix}\right],$$ where ${\bf x}_1,{\bf y}_1$ are $e\times 1$ vectors and ${\bf x}_2,{\bf y}_2$ are $(d-e)\times 1$ vectors. Then \begin{align*} f(w)&= {\bf w}^T\Ab_w{\bf v}\qquad (w={i_1}\cdots {i_s})\\
&=\left[{\bf y}_1^T\ {\bf y}_2^T\right]\left[\begin{matrix}{\bf B}_w & \sum_{j=0}^{s-1}{\bf B}_{i_1}\cdots {\bf B}_{i_j}{\bf D}_{i_{j+1}}{\bf C}_{i_{j+2}}\cdots{\bf C}_{i_s}\\ {\bf 0}_{(d-e)\times e} & {\bf C}_w  \end{matrix}\right]\left[\begin{matrix} {\bf x}_1\\ {\bf x}_2\end{matrix}\right]\\
&={\bf y}_1^T{\bf B}_w{\bf x}_1+{\bf y}_2^T{\bf C}_w{\bf x}_2+{\bf y}_1^T\left(\sum_{j=0}^{s-1}{\bf B}_{i_1}\cdots {\bf B}_{i_j}{\bf D}_{i_{j+1}}{\bf C}_{i_{j+2}}\cdots{\bf C}_{i_s}\right){\bf x}_2.
\end{align*} By the induction hypothesis, we have that both $g(w):={\bf y}_1^T{\bf B}_w{\bf x}_1$ and $h(w):={\bf y}_2^T{\bf C}_w{\bf x}_2$ are in $\C{R}_0(\gS)$. Thus it suffices to check that $$k(w):={\bf y}_1^T\left(\sum_{j=0}^{s-1}{\bf B}_{i_1}\cdots {\bf B}_{i_j}{\bf D}_{i_{j+1}}{\bf C}_{i_{j+2}}\cdots{\bf C}_{i_s}\right){\bf x}_2$$ is in $\C{R}_0(\gS)$.

To this end, let $f_1,\ldots,f_e:\gS^*\to\B{C}$ be defined by $$f_i(w)={\bf y}_1^T{\bf B}_w {\bf e}_i,$$ where ${\bf e}_i$ is the $e\times 1$ column vector with a $1$ in the $i$th position and zeros in all other positions, and given a word $w={i_1}\cdots {i_s}$, we define the maps $g_1,\ldots,g_e:\gS^*\to\B{C}$ by $$g_i(w)=g_i({i_1}\cdots {i_s})={\bf e}_i^T{\bf D}_{i_1}{\bf C}_{i_2}\cdots{\bf C}_{i_s}{\bf x}_2,$$ where we use the convention that $g_i(\eps)=0$ for $\eps$ the empty word. Then $$k(w)=\sum_{i=1}^e \sum_{w_1w_2=w} f_i(w_1)g_i(w_2)=\sum_{i=1}^e(f_i\star g_i)(w).$$ By the inductive hypothesis $f_1,\ldots, f_e\in\C{R}_0(\gS)$. We now show that $g_i\in\C{R}_0(\gS)$ for each $i\in\{1,\ldots,e\}$. 

To see this, let $i\in\{1,\ldots,e\}$. Recall that by definition, $$g_i(w)=g_i({i_1}\cdots {i_s})={\bf e}_i^T{\bf D}_{i_1}{\bf C}_{i_2}\cdots{\bf C}_{i_s}{\bf x}_2.$$ For $j=1,\ldots,m$, set ${\bf z}_j(w)={\bf e}_i^T{\bf D}_j,$ and define the functions $h_j(w)={\bf z}_j^T {\bf C}_w {\bf x}_2,$ and $$\phi_j(w)=\begin{cases} 1 & \mbox{if $w=j$}\\ 0 & \mbox{otherwise.}\end{cases}$$ It is immediate that $\phi_j$ is $\gS$-automatic for each $j=1,\ldots,m$. Thus for each $j$ we have $\phi_j\in\C{R}_0(\gS_m)$. By the inductive hypothesis, each $h_j\in\C{R}_0(\gS_m)$, and so since $$\sum_{j=1}^m (\phi_j\star h_j)(w)=g_i(w),$$ we have that $g_i\in \C{R}_0(\gS_m)$. By definition $\C{R}_0(\gS_m)$ is closed under the convolution product and the taking of finite linear combinations, thus we have that $k(w)\in\C{R}_0(\gS_m)$. So $\C{R}_0(\gS_m)\supseteq \C{S}_0(\gS_m)$.

Since $\C{S}_0(\gS_m)$ contains all of the $(\B{C},m)$-automatic sequences, as given by Lemma~\ref{gsauto}, to show that $\C{R}_0(\gS_m)\subseteq \C{S}_0(\gS_m)$, it is sufficient to prove that $\C{S}_0(\gS_m)$ is closed under taking $\mathbb{C}$-linear combinations and under the convolution product. 

It is quite clear that if $\lambda\in \mathbb{C}$ and $f,g\in\C{S}_0(\gS_m)$ then $f+\lambda g\in\C{S}_0(\gS_m)$. For if \begin{equation}\label{fandg}f(w)={\bf w}_1^T \Ab_w {\bf v}_1,\qquad\mbox{and}\qquad g(w)={\bf w}_2^T {\bf B}_w {\bf v}_2,\end{equation} then $$(f+\lambda g)(w)=\left[{\bf w}_1^T\ {\bf w}_2^T\right]\left[\begin{matrix}\Ab_w & {\bf 0}\\ {\bf 0} & {\bf B}_w  \end{matrix}\right]\left[\begin{matrix} {\bf v}_1\\ \lambda {\bf v}_2\end{matrix}\right],$$ and the eigenvalues of $$\left[\begin{matrix}\Ab_w & {\bf 0}\\ {\bf 0} & {\bf B}_w  \end{matrix}\right]$$ are just the eigenvalues of $\Ab_w$ and ${\bf B}_w$, which are in $\C{U}\cup\{0\}.$ Hence $\C{S}_0(\gS_m)$ is closed under addition.

To see that $\C{S}_0(\gS_m)$ is a $\B{C}$-algebra under the convolution product, let $f,g\in\C{S}_0(\gS_m)$ be given as in \eqref{fandg}. Let $${\bf C}_{i,j}:=\left[\begin{matrix}\Ab_i & \gd_{i,j} {\bf v}_1{\bf w}_2^T {\bf B}_i\\ {\bf 0} & {\bf B}_i  \end{matrix}\right],$$ where $$\gd_{i,j}:= \begin{cases} 1 &\mbox{if $i=j$}\\ 0 &\mbox{if $i\neq j$}.\end{cases}$$ Then all eigenvalues from matrices in the semigroup generated by $\langle {\bf C}_{1,j},\ldots,{\bf C}_{m,j}\rangle$ lie in $\C{U}\cup\{0\}$. Thus if we write ${\bf C}_{w,j}={\bf C}_{i_1,j}\cdots {\bf C}_{i_s,j}$ when $w=i_1\cdots i_s$, we have \begin{multline*} h_j(w):=\left[{\bf w}_1^T\ {\bf 0}\right]{\bf C}_{w,j}\left[\begin{matrix} {\bf 0}\\ {\bf v}_2\end{matrix}\right]\\
= \sum_{\ell=0}^{s-1}\left({\bf w}_1^T \Ab_{i_1}\cdots \Ab_{i_\ell}{\bf v}_1{\bf w}_2^T {\bf B}_{i_{\ell+1}}\cdots {\bf B}_{i_s}{\bf v}_2\right)\gd_{j,i_{\ell+1}}\\
=\sum_{\ell=0}^{s-1}f(x_{i_1}\cdots x_{i_\ell})g(x_{i_{\ell+1}}\cdots x_{i_s})\gd_{j,i_{\ell+1}}\in\C{S}_0(\gS_m).
\end{multline*} Thus $(f\star g)(w)=h_1(w)+\cdots+h_m(w)+g(\eps)f(w)\in\C{S}_0(\gS_m).$
\end{proof}

\begin{proof}[Proof of Theorem \ref{iff}]
We have shown the equivalence of (i) and (ii) in Proposition~\ref{Thm1}.  We now show these are both equivalent to (iii).  

We first show that (ii) implies (iii).  Suppose that $f\in \C{S}_0(\gS_m)$.  Then there exist positive integers $m$ and $d$, matrices 
    $\Ab_1,\ldots,\Ab_m\in \B{C}^{d\times d}$, and vectors ${\bf v},{\bf w}\in \B{C}^d$ 
    such that 
    $$f(w)={\bf w}^T \Ab_w {\bf v},$$ 
    where $\Ab_w:=\Ab_{i_1}\cdots\Ab_{i_s}$, when $w={i_1}\cdots {i_s}.$  Moreover, by definition of $\C{S}_0(\gS_m)$, we may assume that the semigroup generated by  the $\Ab_i$ is tame.  
    Since we have a tame matrix semigroup, Theorem 1.2 of \cite{B2005} gives that there is a $c>0$ and a $k>0$ such that $|\Ab_w|\leqslant c\cdot|w|^k$ for all $w$ of length at least $1$.  It then follows
    $|f(w)| \leqslant c \|{\bf w}\| \cdot \|{\bf v}\| \cdot |w|^k$ for all non-trivial words $w$ and so we get that the growth of $f$ is polynomially bounded.   
 
It remains to show that (iii) implies (ii); we show this by proving the  contrapositive. We present an argument similar to the proof of Theorem 1.1 of \cite{BCH2014}. 

As in the proof of Lemma \ref{gsauto}, we can assume without loss of generality that $\spn_{\B{C}}{\bf w}^T\langle \C{A}\rangle=\B{C}^{1\times
     d}$ and $\spn_{\B{C}}\langle\C{A}\rangle{\bf v}=\B{C}^{d\times 1}.$ 
     
Suppose that $f:\Sigma_m^*\to\B{Z}$ is not in $\C{S}_0(\Sigma_m)$ and let ${\bf w},{\bf v},$ and $\C{A}$ be as given by the canonical representation of $f$, and recall (by Lemma \ref{canonical}) that $\spn_{\B{C}}{\bf w}^T\langle \C{A}\rangle=\B{C}^{1\times d}$. 
Then there is a word  $w\in\Sigma_m^*$ such that the matrix ${\bf A}_w\in\langle \C{A}\rangle $ has an eigenvalue that is nonzero and not a root of unity.  By Kronecker's theorem, ${\bf A}_w$ then has an eigenvalue $\lambda$ of modulus strictly larger than one.
Thus ${\bf A}_w$ has an eigenvector ${\bf y}$ such that ${\bf A}_w{\bf y}=\lambda{\bf y}.$ Now pick a nonzero vector ${\bf x}$ such that ${\bf x}^T{\bf y}=C\neq 0$; as in Lemma \ref{canonical} this vector can be taken to be the vector of complex conjugates of ${\bf y}$. Then $|{\bf x}^T{\bf A}_w^n{\bf y}|=|C|\cdot |\lambda|^n.$ 

Since $\spn_{\B{C}}{\bf w}^T\langle \C{A}\rangle=\B{C}^{1\times d}$ there exist an integer $\ell$,  words $x_1,\ldots,x_\ell\in\Sigma_m^*$, and complex numbers $\ga_1,\ldots,\ga_\ell,$ such that $${\bf x}^T=\sum_{i=1}^\ell \ga_i{\bf w}^T {\bf A}_{x_i}.$$ If we write ${\bf y}=[\beta_1,\ldots,\beta_d]^T$, then \begin{align} \nonumber |C|\cdot|\lambda|^n=|{\bf x}^T{\bf A}_w^n{\bf y}|&=\left|\sum_{i=1}^\ell\ga_i{\bf w}^T{\bf A}_{x_iw^n}{\bf y}\right|\\
\nonumber&=\left|\sum_{i=1}^\ell\sum_{j=1}^d\ga_i\gb_j g_j(x_iw^n)\right|\\
\label{lambdaleq}&\leqslant\sum_{i=1}^\ell\sum_{j=1}^d|\ga_i\gb_j|\cdot |g_j(x_iw^n)|.
\end{align} 

Now each of the basis functions $g_j$ is in the $\B{Z}$-module spanned by the maps $f^u$, so that, by \eqref{lambdaleq}, there exist positive integers $H$ and $Q$, integers $\{\gamma_{h,q}\}_{h\leqslant H,q\leqslant Q}$ (not all zero), and words $u_1,\ldots,u_H\in\Sigma^*$, such that \begin{equation}\label{lambdaleqf}|C|\cdot|\gl|^n\leqslant\sum_{i=1}^\ell\sum_{j=1}^d\sum_{h=1}^H\sum_{q=1}^Q|\ga_i\beta_j\gamma_{h,q}|\cdot|f(u_hx_iw^n)|.\end{equation}

Let $K=\sum_{i=1}^\ell\sum_{j=1}^d\sum_{h=1}^H\sum_{q=1}^Q|\ga_i\beta_j\gamma_{h,q}|$. Thus  \eqref{lambdaleqf} implies that some element from
$$\big\{ \{|f(u_h x_i w^n)|\}_{n\geqslant 0} \colon i=1,\ldots,\ell\ \mbox{and}\ h=1,\ldots,H\}\big\}$$ is at least $|C|\cdot|\lambda|^n/K$.  

We let $M$ denote the maximum of the lengths of $x_1,\ldots ,x_\ell,w,u_1,\ldots ,u_h$.
Then for each $i=1,\ldots,\ell$ and $h=1,\ldots,H$ we have $|u_h x_i w^n|< 2Mn$ for $n\geqslant 2$.   Hence we have constructed an infinite set of words $u_h x_i w^n$ such that 
$$|f(u_h x_i w^n)|\geqslant \frac{|C|}{K}\cdot|\lambda|^n>\frac{|C|}{K}\cdot|\lambda|^{\frac{|u_hx_iw^n|}{2M}}.$$ Since $|\lambda|>1$, we have that $|\lambda|^{1/2M}>1$, and thus $f$ is not polynomially bounded. This finishes the proof of the theorem.
\end{proof}

\begin{proof}[Proof of Theorem \ref{main}]
We prove the stronger result that whenever $f:\gS_m^*\to \mathbb{C}$ is in $\C{S}_0(\Sigma_m)$, then ${\rm GrDeg}(f)$ is a nonnegative integer.  Since $f\in \C{S}_0(\gS_m)$, there exist positive integers $m$ and $d$, matrices 
    $\Ab_1,\ldots,\Ab_m\in \B{C}^{d\times d}$, and vectors ${\bf v},{\bf w}\in \B{C}^d$ 
    such that 
    $$f(w)={\bf w}^T \Ab_w {\bf v},$$ 
    where $\Ab_w:=\Ab_{i_1}\cdots\Ab_{i_s}$ when $w={i_1}\cdots {i_s}.$  Moreover, by the definition of $\C{S}_0(\gS_m)$, we may assume that the semigroup $\SG$ is tame.  
    
   We may also assume that the complex span of ${\bf w}^T \Ab_w$ as $w$ ranges over $\gS_m^*$ is all of $\mathbb{C}^{1\times d}$ and that the complex span of $ \Ab_w{\bf v}$ as $w$ ranges over $\gS_m^*$ is all of $\mathbb{C}^{d\times 1}$; otherwise, we can work with a smaller representation that has this property and the resulting semigroup will still be tame.  Let $\| \cdot \|$ be the norm given by $\|{\bf A}\| = {\rm Tr}({\bf A}{\bf A}^*)$.  By Theorem \ref{morder}, there exists some nonnegative integer $k$ and positive constants $C_1$ and $C_2$ such that
   $C_1 |w|^k \leqslant \|\Ab_w\| \leqslant C_2 |w|^k$ for all non-trivial words $w$.  In particular, we have
   $|f(w)| \leqslant C_2\cdot \|{\bf w}\|\cdot \|{\bf v}\|\cdot |w|^k$ for all non-trivial words $w$ and so ${\rm GrDeg}(f)\leqslant k$.  On the other hand, the fact that $ \|\Ab_w\|\geqslant C_1 |w|^k$ gives that there is some $C_3$ and some $i,j$ such that $|{\bf e}_i^T \Ab_w {\bf e}_j| \geqslant C_3 |w|^k$ for infinitely many words $k$.  Since ${\bf e}_i^T$ is in the span of ${\bf w}^T \Ab_u$ for a fixed finite set of words $u$ and ${\bf e}_j$ is in the span of $\Ab_{u'} {\bf v}$ for a fixed finite set of words $u'$, we get that
   $${\bf e}_i^T \Ab_w {\bf e}_j$$ is a fixed linear combination of elements of the form $f(uwu')$ with $u,u'$ running over a finite set.  It follows that there exist fixed $u_0$ and $u_0'$ such that $|f(u_0wu_0')|\geqslant C_3 |w|^k$ for infinitely many $w$.  Since $u_0$ and $u_0'$ are fixed, there is some $C_4>0$ such that
  $ |f(u_0wu_0')|\geqslant C_4 |u_0u_0'w|^k$ for infinitely many $w$ and so   
$ {\rm GrDeg}(f)\geqslant k$. Thus the growth degree is precisely $k$.
\end{proof}

\bibliographystyle{amsplain}
\providecommand{\bysame}{\leavevmode\hbox to3em{\hrulefill}\thinspace}
\providecommand{\MR}{\relax\ifhmode\unskip\space\fi MR }
\providecommand{\MRhref}[2]{%
  \href{http://www.ams.org/mathscinet-getitem?mr=#1}{#2}
}
\providecommand{\href}[2]{#2}


\end{document}